\documentclass[12pt]{svjour3}
\usepackage{amsfonts}
\usepackage{latexsym}
\usepackage{amsmath}
\usepackage{amssymb}
\usepackage{amscd}
\usepackage{epsfig}
\addtocounter{MaxMatrixCols}{4}
\newcommand{\undertilde}[1]{\ensuremath{\mathord{\vtop{\ialign{##\crcr
   $\hfil\displaystyle{#1}\hfil$\crcr\noalign{\kern1.5pt\nointerlineskip}
   $\hfil\tilde{}\hfil$\crcr\noalign{\kern1.5pt}}}}}}
\begin{document}

\title{On the Determinant and its Derivatives of the Rank-one Corrected Generator of a Markov Chain on a Graph.}
\author{J.A. Filar, M. Haythorpe, W. Murray}
\institute{J.A. Filar
\at Flinders University\\
\email{jerzy.filar@flinders.edu.au}
\and
M. Haythorpe
\at Flinders University\\
\email{michael.haythorpe@flinders.edu.au}
\and
W. Murray
\at Stanford University\\
\email{walter@stanford.edu}
}
\maketitle {\abstract We present an algorithm to find the determinant and its
first and second derivatives of a
rank-one corrected generator matrix of a doubly stochastic Markov chain.
The motivation arises from the fact that the global minimiser of this
determinant solves the Hamiltonian cycle problem. It is essential for
algorithms that find global minimisers to evaluate both first and second
derivatives at every iteration. Potentially the computation of these
derivatives could require an overwhelming amount of work since  for the
Hessian $N^2$ cofactors are required. We show how the  doubly stochastic
structure and the properties of the objective may be exploited to calculate all cofactors from a single LU
decomposition.}

\vspace*{-1cm}
\section{Introduction}\label{sec-Introduction}
\vspace*{-0.5cm}

The {\em Hamiltonian cycle problem} (HCP) is an important graph theory problem that features prominently in
complexity theory because it is {\em NP-complete} \cite{garey1976}.
HCP has also gained recognition because two
special cases: the {\em Knight's tour} and the {\em Icosian game}, were
solved by Euler and Hamilton,
respectively. Finally, HCP is closely related to the well known Traveling Salesman problem.

The definition of HCP is the following: given a graph $\Gamma$ containing $N$ nodes,
determine whether any simple cycles of length $N$ exist in the graph.
These simple cycles of length $N$ are known as {\em Hamiltonian cycles}.
If $\Gamma$ contains at least one Hamiltonian cycle (HC), we say that
$\Gamma$ is a {\em Hamiltonian graph}. Otherwise, we say that $\Gamma$
is a {\em non-Hamiltonian graph}.

While there are many graph theory techniques that have been designed to solve HCP, another common approach is to
associate a variable $x_{ij}$ with each arc $(i,j) \in \Gamma$, and solve an associated optimisation problem. A convenient method of representing these constraints is to use a matrix $P({x})$, where
\begin{eqnarray}p_{ij}({x}) & = & \left\{\begin{array}{lcl}x_{ij}, & & (i,j) \in \Gamma,\\
0, & & \mbox{otherwise.}\end{array}\right.\nonumber\end{eqnarray}
The discrete nature of HCP naturally lends itself to integer programming optimisation problems. However, arising from an embedding of HCP in a Markov decision process (a practice initiated by Filar et al \cite{krass}), continuous optimisation problems that are equivalent to HCP have been discovered in recent times. In particular, it was demonstrated in \cite{detpaper} that if we define
\begin{eqnarray}A(P({x})) & = & I - P({x}) + \frac{1}{N}ee^T,\label{eq-APx}\end{eqnarray}
where $e$ is a column vector with unit entries, then HCP is equivalent to solving the following optimisation problem:
$$\min \quad-\det(A(P({x})))$$
subject to
\begin{eqnarray}\sum\limits_{j \in \mathcal{A}(i)}x_{ij} & = & 1, \quad i = 1, \hdots, N,\label{eq-DS1}\\
\sum\limits_{i \in \mathcal{A}(j)}x_{ij} & = & 1, \quad j = 1, \hdots, N,\label{eq-DS2}\\
x_{ij} & \geq & 0, \quad\forall (i,j) \in \Gamma,\label{eq-DS3}\end{eqnarray}
where $\mathcal{A}(i)$ is the set of nodes reachable in one step from node $i$. Constraints (\ref{eq-DS1})--(\ref{eq-DS3}) are called the {\em doubly-stochastic} constraints. For neatness, we refer to constraints (\ref{eq-DS1})--(\ref{eq-DS3}) as the set $\mathcal{DS}$, and we call the objective function $f(P({x}))$. Then, the above problem can be represented as follows:

\begin{eqnarray}
\min \{f(P({x}))\;\;\; | \;\;\;{x} \in \mathcal{DS}\}.\label{eq-problem_constrained}
\end{eqnarray}

Note that we need to find a global minimiser and, typically, there are many of them. However, the number of global minimisers is typically extremely small compared to the number of local minimisers. One consequence of multiple global minimisers is that there is a similarly large number of stationary points. To distinguish such stationary points from minimisers, algorithms (see for example \cite{mike}) require the use of second derivatives. It is not hard to appreciate that evaluating such derivatives will be expensive even for moderately-sized problems unless some special structure is identified. We exploit the structure of the Hessian and the fact the points at which it is evaluated are in $\mathcal{DS}$. We assume an algorithm to solve this problem starts at a feasible point and all iterates remain feasible. For problems with nonlinear objectives this is almost always the best approach. Finding a point in $\mathcal{DS}$ is simple and does not involve evaluating the objective or its derivatives.

\vspace*{-0.85cm}
\section{Preliminary results}\label{sec-Preliminary results}
\vspace*{-0.5cm}
Our work is based on some key properties of the LU factorisation of the matrix $I - P$. We adopt the following common notation. As noted we denote a vector of unit elements by $e$. In a given proof $e$ will still be used even when different occurrences may be of different dimension since its dimension can be inferred. The vector $e_j$ has all zero elements except the $j$th, which is unity. We will often denote the $(i, j)$th elements of a matrix, say $W$, by $w_{ij}$.

\begin{definition}A matrix $G$ is said to have property $S$ if
\begin{eqnarray}
    g_{ii} & \geq & 0, \quad \forall i = 1, \hdots, N,\nonumber\\
    g_{ij} & \leq & 0, \quad \forall i \neq j,\nonumber\\
    Ge = G^Te & = & 0.\nonumber\end{eqnarray}
\end{definition}
\begin{theorem}
    If $G$ has property $S$ the LU decomposition of $G$ exists.
\end{theorem}
\begin{proof}
Following common convention the diagonal elements of $L$ are 1. There is no loss of generality if we assume that $g_{11} \neq 0$. After one step of Gaussian elimination (GE) we obtain
$$
    G = \left(\begin{array}{cc}1 & { 0}^T\\
    { \ell} & I\end{array}\right)\left(\begin{array}{cc}1 & { 0}^T\\
    { 0} & \hat{G}\end{array}\right)\left(\begin{array}{cc}u_{11} & {u}^T\\
    { 0} & I\end{array}\right).
$$
Note that, expanding this matrix multiplication gives
$$
    G = \left(\begin{array}{cc}u_{11} & {u}^T\\
    u_{11}{ \ell} &  \;\;\hat{G} +{ \ell u^T}  \end{array}\right),
$$
and therefore ${ u^T} = \left[\begin{array}{cccc}g_{12}\; & \;g_{13}\; & \cdots\; & \;g_{1N}\end{array}\right]$ and ${ \ell} = \displaystyle\frac{1}{u_{11}}\left[\begin{array}{cccc}g_{21}\; & \;g_{31}\; & \cdots\; & \;g_{N1}\end{array}\right]^T$. Clearly, $u_{11} = g_{11} \geq 0$, and therefore ${ u}^T \leq 0$ and ${ \ell} \leq 0$.
We have $G{ e} = { 0}$, which implies
$$\left(\begin{array}{cc}1 & { 0}^T\\
{ 0} & \hat{G}\end{array}\right)\left(\begin{array}{cc}u_{11} & { u}^T\\
{ 0} & I\end{array}\right){ e} = { 0}.$$
It follows that $\hat{G}{ e} = \hat{G}^T{ e} = { 0}$. By definition we have
$$\hat{g}_{ij} = g_{i+1,j+1} - \ell_iu_j, \quad \forall i \neq j.$$
Since $l_i \leq 0$ and $u_j \leq 0$ it follows that $\hat{g}_{ij} \leq 0,\;\forall i \neq j$. From this result and $\hat{G}{ e} = { 0}$ it follows that $\hat{g}_{ii} \geq 0$ and that $\hat{G}$ has property $S$. We can now proceed with the next step of GE. Note that if $\hat{g}_{11} = 0$ we must have the first row and column of $\hat{G}$ be zero, which implies the corresponding row of $U$ is zero and the off-diagonal elements of $L$ are zero. Consequently, we can proceed until the leading diagonal element of $\hat G$ is nonzero.
\end{proof}

\begin{lemma} \label{lem-UNN}
        Regardless of the rank of $G$  we have $u_{NN} = 0$.
\end{lemma}
\begin{proof}
    After $N-1$ steps of GE $\hat G$ is a $1\times 1$ matrix. The only $1 \times 1$ matrix that satisfies property $S$ is $0$.
\end{proof}

\begin{corollary}If $G$ satisfies property $S$ and $G = LU$, then $Ue = {0}$.\end{corollary}

\begin{proof}The result follows immediately from the nonsingularity of $L$ and the property that $Ge = 0$.\end{proof}

\begin{lemma}If $G$ satisfies property $S$ and $G = LU$, the off-diagonal elements of $L$ and $U$ are non-positive. Moreover, $ -1 \le l_{ij} \le 0$ for $i \ne j$. \label{lem-LU}
\end{lemma}

\begin{proof}

The first part follows immediately from the fact that ${u}^T \leq 0$ and ${ l} \leq 0$ in each iteration of GE. Note that when performing the GE the pivot used is always the largest in magnitude. This is a consequence of $\hat G^T e = 0$ and the off-diagonal elements in the column of $\hat G$ being negative. Since the elements being eliminated are not bigger in magnitude than the pivot it follows that the magnitude of elements of $L$ are not bigger than 1.

\end{proof}

The above result has important consequences for the condition number of $L$.

\begin{lemma}
If $G$ satisfies property $S$, is  rank $N - 1$ and $G = LU$, then $L^T e = e_N$.
\end{lemma}

\begin{proof}Since $U$ has rank $N-1$ and $u_{NN} = 0$, the basis for the nullspace of $U^T$ is $\kappa e_N$, where $\kappa \ne 0$. From property $S$,

$$G^T e = U^T (L^T e) = 0$$

and hence we obtain $L^T e = \kappa e_N$, for some $\kappa$. Since $L$ has unit diagonal entries, we have $\kappa = 1$, which yields the result.\end{proof}

%
%
%
%
%
%

\begin{lemma}
    When $G$  satisfies property $S$ and is rank $N-1$ then $ \bar G \equiv G + { e}{ e}_N^T$ is nonsingular. \label{lem-Ubar}
\end{lemma}

\begin{proof}

We have
$$
    \bar G= LU +  ee^T_N = L(U + v e^T_N)= L\bar U,
$$
where $L v = e$. Note that $\bar U = U +v e^T_N$ is upper triangular. Moreover,
the $(N,N)$th element of $\bar U$ is $v_N$. Since
\begin{eqnarray}e^T Lv = e_N^T v = v_N = e^T e = N,\label{eq-eTLv}\end{eqnarray}
it implies that $\bar{U}$, and hence $\bar G$, is nonsingular.
\end{proof}

Clearly $G = I-P$ has property $S$. It should be noted that the existence of an LU decomposition for $G$ has been demonstrated previously for \emph{irreducible} $P$ in Heyman \cite{hey1}. In the context of this paper, we note that when $P(x)$ corresponds to the strict interior of $\mathcal{DS}$, irreducibility applies.

\vspace*{-0.8cm}
\section{Computing $\det(A(P))$ using an LU decomposition}\label{sec-LU Decomposition}
\vspace*{-0.4cm}

When appropriate we suppress the (fixed) argument ${x}$ in $P({x})$, $A(P({x}))$ and $f(P({x}))$, and write simply $P$, $A(P)$ (or just $A$) and $f(P)$, respectively. The most efficient way to compute the determinant of a matrix is to compute its LU decomposition. Normally we obtain $LU = \Pi A$, where $ \Pi$ is a permutation matrix. However, we shall show that det ($A$) may be computed from either the determinant of  the leading principal minor of  $I - P$ or of det($I - P + ee_N^T)$. Consequently, no permutation matrix is needed. This is of importance since unlike $A$ we expect $I - P$ to be sparse, which is the case of interest for finding HC.  Factorising sparse matrices is much faster than factorising dense matrices. The difference is even greater when it is unnecessary to do numerical pivoting. Knowing apriori where the fill-in is in the factors will enable more efficient data management and avoid indirect addressing.

We assume an algorithm to solve (\ref{eq-problem_constrained}) starts at a feasible point and all iterates remain feasible. For problems with nonlinear objectives this is almost always the best approach. Finding a point in $\mathcal{DS}$ is simple and does not involve evaluating the objective or its derivatives. It is known that the global minimiser of (\ref{eq-problem_constrained}) is $-N$. Indeed, from Theorem 4.1, Proposition 4.4 and Proposition 4.6 in Ejov et al \cite{detpaper}, we know that $-\det(A(P(x))) \in [-N,0]$, where the upper bound is obtained when $A(P(x))$ is singular. However, we have no interest in finding the determinant of $A$ when it is singular. For the rest of this paper we assume $x \in \mathcal{DS}$ and that $A(P(x))$ is nonsingular. Since $A(P(x))$ is merely a rank-one correction of $I - P$, and the latter is always singular for $x \in \mathcal{DS}$, the nonsingularity of $A(P(x))$ implies that $I - P$ has rank $N - 1$.


\subsection{Product forms of $A(P)$ and $\det(A(P))$}\label{subsec-det}
To calculate the objective function $f(P)$, its gradient and Hessian, we begin by performing an LU decomposition
to obtain
\begin{eqnarray}LU & = & G = I - P.\label{eq-lu}\end{eqnarray}
 While there may be a need to reorder the matrix to obtain a sparse factorisation. Since the pattern of nonzero elements in $G$ is symmetric then symmetric pivoting may be used, which is equivalent to renumbering the nodes of the graph. It follows we may assume the effort to compute the LU factors is $O(N^2)$. Note this reordering is done once since it depends only on the location of nonzeros and not their values. Whatever ordering is chosen the LU factors of $G$ reordered exist without the need for pivoting based on numerical considerations.

The outline of the derivation of $f(P)$ in terms of $U$ is as follows.
\begin{enumerate}
\item[(1)] We express A(P) as the product of three nonsingular factors.
\item[(2)] We show that two of these factors have a determinant of 1.
\item[(3)] We show that the third factor shares all but one eigenvalue with $U$, with the single different eigenvalue being $N$ (rather than 0).
\item[(4)] We express the determinant as a
product of the first $N-1$ diagonal elements of $U$, and $N$.\end{enumerate}
First, we express $A$ as a product of $L$ and another factor. Let ${v}$ be an $N \times 1$ vector, and $\bar{U}$ be an $N \times N$ matrix, such that
\begin{eqnarray}
L{v} = {e}\mbox{,\quad and \;\;\;\;\;}\bar{U} := U + {v}{e}^T_N,\label{eq-lu2}
\end{eqnarray}
where ${e}_N^T = \left[\begin{array}{cccc}0 & \cdots & 0 & 1\end{array}\right]$.
Since $L$ is
nonsingular, ${v} \ne 0$ exists and is unique, and therefore $\bar{U}$ is well-defined. The first $N - 1$ columns of $\bar U$ are identical to those of $U$.  Consequently $\bar U$ is also upper triangular. Since $LU = G$ satisfies property $S$, it follows from Lemma \ref{lem-UNN} that $u_{NN} = 0$, and therefore from (\ref{eq-eTLv}) that $\bar u_{NN} = v_N = N$.
 Therefore,
\begin{eqnarray}
{\hfill}\det\bar{U} =  \prod_{i=1}^N \bar{u}_{ii} =
N \prod_{i=1}^{N-1} u_{ii},
\label{eq-detUbar}
\end{eqnarray}

where $u_{ii}$ is the $i$-th diagonal element of $U$. Exploiting (\ref{eq-lu2}) we may write
\begin{eqnarray} A & = & (I - P) + \frac{1}{N}{e}{e}^T \;\;=\;\; LU + \frac{1}{N}L{v}{e}^T \;\;=\;\; L(U + \frac{1}{N} {v} {e}^T)\nonumber\\
& = & L(\bar{U}
+ {v}[\frac{1}{N} {e}^T - {e}_N^T]).\label{eq-A4}\end{eqnarray}

Since $\bar U$ is nonsingular we may define ${w}$ to be the unique solution to the system
\begin{eqnarray}\bar{U}^T {w} & = & \frac{1}{N} {e} - {e}_N.\label{eq-U^Tw}\end{eqnarray}
Then, from (\ref{eq-A4})

\begin{eqnarray}A = L\left(\bar{U} + {v}\left[\frac{1}{N}{e}^T - {e}_N^T\right]\right) = L\left(\bar{U} + {v}{w}^T\bar{U}\right) = L\left(I + {v}{w}^T\right)\bar{U}.\label{eq-A5}\end{eqnarray}
We take the determinant of (\ref{eq-A5}) to obtain
\begin{eqnarray}\det(A) & = & \det(L)\det(I + {v} {w}^T)\det(\bar{U}).\label{eq-detA}\end{eqnarray}
Note that, for any vectors ${ c}$ and ${ d}$,
\begin{eqnarray}\det(I + { c}{d}^T) & = & 1 + {d}^T{c}.\label{eq-identity with c and d}\end{eqnarray}
This is because ${c}{d}^T$ has one eigenvalue ${d}^T{c}$ of multiplicity 1 and an eigenvalue 0 of multiplicity $N-1$. Consequently,
\begin{eqnarray}\det(I + {v}{w}^T) & = & 1 + {w}^T{v},\label{eq-detrankcorrection}\end{eqnarray}
which we substitute into expression (\ref{eq-detA}) above.

\begin{lemma}\label{lem-w^Tv = 0}The inner-product ${w}^T{v}$ in (\ref{eq-detrankcorrection}) satisfies
\begin{eqnarray}{w}^T {v} & = & 0.\nonumber\end{eqnarray}
\end{lemma}
\begin{proof}From their respective definitions (\ref{eq-U^Tw}) and (\ref{eq-lu2}),
\begin{eqnarray}{w}^T & = & \left(\frac{1}{N} {e}^T - {e}_N^T\right)(\bar{U})^{-1},\label{eq-wT}\\
{v} & = & L^{-1}{e}.\label{eq-v}\end{eqnarray}
Then, from (\ref{eq-wT})--(\ref{eq-v}) we obtain
\begin{eqnarray}{w}^T {v} & = & \left(\frac{1}{N} {e}^T - {e}_N^T\right)(\bar{U})^{-1}L^{-1}{e} \;\;=\;\; \left(\frac{1}{N} {e}^T - {e}_N^T\right)(L\bar{U})^{-1}{e}\nonumber\\
& = & \left(\frac{1}{N} {e}^T - {e}_N^T\right)(I - P + {e}{e}_N^T)^{-1}{e}.\label{eq-construction of wTv}\end{eqnarray}
Since $I$, $P$ and ${e}{e_N}^T$ are all stochastic matrices, we know that $I - P + {e}{e}_N^T$ has row sums of 1 as well. Hence, its inverse also has row sums equal to 1, that is,
\begin{eqnarray}(I - P + {e}{e}_N^T)^{-1}{e} & = & {e}.\label{eq-inverse row sums}\end{eqnarray}
Substituting (\ref{eq-inverse row sums}) into (\ref{eq-construction of wTv}), we obtain
\begin{eqnarray*}\hspace*{3.2cm}{w}^T {v} = \left(\frac{1}{N} {e}^T - {e}_N^T\right){e} = 0,\end{eqnarray*}
which concludes the proof.\end{proof}
We now derive the main theorem of this subsection.
\begin{theorem}\label{theorem-det}
Let $LU$ denote the LU decomposition of $I - P$ and $u_{ii}$ be the diagonal elements of $U$ then
\begin{eqnarray}\det(A(P)) & = & N \prod\limits_{i=1}^{N-1} u_{ii}. \nonumber
\end{eqnarray}
\end{theorem}
\begin{proof}
From (\ref{eq-detA}), (\ref{eq-detrankcorrection}) and Lemma \ref{lem-w^Tv = 0} we know that
\begin{eqnarray}\det(A(P)) & = & \det(L)(1 + 0)\det(\bar{U}).\nonumber\end{eqnarray}
From the construction of the LU decomposition we know that $\det(L) = 1$ and using (\ref{eq-detUbar})
gives
\begin{eqnarray}\det(A(P)) & = & \det(\bar{U})= N\prod\limits_{i=1}^{N-1}
u_{ii}.\nonumber\end{eqnarray}
This concludes the proof.\end{proof}
\begin{remark}Note that finding ${v}$ and ${w}$ is a simple process because $L$ and $\bar{U}^T$ are lower-triangular matrices, so we can solve the systems of linear equations in (\ref{eq-U^Tw}) and (\ref{eq-lu2}) directly. \end{remark}
\subsection{Finding the gradient {\em g(P)}}\label{subsec-gradient ipm}
Next we use the LU decomposition found in Subsection \ref{subsec-det} to find the gradient of
$f(P) = -\det(A(P))$. Note that since variables of $f(P)$ are entries $x_{ij}$ of the probability transition
matrix $P({x})$, we derive an expression for $g_{ij}(P) := \frac{\partial f(P)}{\partial x_{ij}}$ for each
$x_{ij}$ such that $(i,j) \in \Gamma$.

Consider vectors $a_j$ and $b_i$ satisfying the equations $\bar{U}^Ta_j = {e}_j$ and $Lb_i = {e}_i$, where ${e}_j$ is a zero vector except for a unit entry in the $j$-th column. Then, we define $Q := I - {v}{w}^T$, where ${v}$ and ${w}^T$ are as in (\ref{eq-wT})--(\ref{eq-v}). We prove the following result
in this subsection:
\begin{eqnarray}g_{ij}(P) = \det(A(P))(a_j^TQb_i),\label{eq-g_ij(P)}\end{eqnarray}
where $g_{ij}(P)$ is the gradient vector element corresponding to the arc $(i,j) \in \Gamma$.

The outline of the derivation of (\ref{eq-g_ij(P)}) is as follows.
\begin{enumerate}\item[(1)] We represent each element $g_{ij}(P)$ of the gradient vector as a cofactor of $A(P)$.
\item[(2)] We construct an elementary matrix that transforms matrix $A(P)$ into a matrix with determinant equal to the above cofactor of $A(P)$.
\item[(3)] We then express the element $g_{ij}(P)$ of the gradient vector as the product of $\det(A(P))$ and the determinant of the elementary matrix, the latter of which is shown to be equal to $a_j^TQb_i$.\end{enumerate}
For any matrix $V = \left(v_{ij}\right)_{i,j=1}^{N,N}$ it is well-known (e.g., see May \cite{May1}) that
$\displaystyle\frac{\partial \det(V)}{\partial v_{ij}} = (-1)^{i+j}\det(V^{ij})$, where $V^{ij}$ is the $(i,j)$-th
minor of $V$. That is, $\displaystyle\frac{\partial \det(V)}{\partial v_{ij}}$ is the $(i,j)$-th cofactor of $V$. Since
the $(i,j)$-th entry of $A(P)$ is simply $a_{ij} = \delta_{ij} - x_{ij} + \frac{1}{N}$ (where $\delta_{ij}$ is the Kronecker delta that is 1 if $i = j$, and 0 otherwise), it now follows that
\begin{eqnarray}g_{ij}(P) = \frac{\partial f(P)}{\partial x_{ij}} = \frac{\partial \left[-\det{A(P)}\right]}{\partial a_{ij}}\frac{da_{ij}}{dx_{ij}} = (-1)^{i+j}\det\left(A^{ij}(P)\right).\label{eq-g}\end{eqnarray}
However, rather than finding the cofactor we calculate (\ref{eq-g}) by finding
the determinant of a modification of $A$ in which the $i$th row  has been replaced with ${e}_j^T$. Since $A$ is a full-rank matrix, it is
possible to perform row operations to achieve this. Suppose $A$ is composed of rows $r_1^T$, $r_2^T$, $\hdots$, $r_N^T$. Then, we perform the following row operation:
\begin{eqnarray}r_i^T & \rightarrow & \alpha_j(1)r_1^T + \alpha_j(2)r_2^T + \hdots +
\alpha_j(N)r_N^T,\label{eq-ipm row operation}\end{eqnarray}
where $\alpha_j(i)$ is the $i$-th element of vector $\alpha_j$ and $A^T\alpha_j =
{e}_j$. This row operation replaces the $i$-th row of $A$ with $\alpha_j^T A = {e}_j^T$, as desired.

In this case, from (\ref{eq-A5}), $A^T = \bar{U}^T(I + {w}{v}^T)L^T$. Since $A$ is nonsingular
$\alpha_j$ can be found directly:
\begin{eqnarray}\alpha_j & = & \left(A^T\right)^{-1}{e}_j \;\;=\;\; \left[\bar{U}^T(I + {w}{v}^T)L^T\right]^{-1}{e}_j \nonumber\\ & = & (L^T)^{-1}(I + {w}{v}^T)^{-1}(\bar{U}^T)^{-1}{e}_j.\label{eq-alpha}\end{eqnarray}
\begin{lemma}
\begin{eqnarray}\left(I + {w}{v}^T\right)^{-1} & = & I - {w}{v}^T.\nonumber\end{eqnarray}\end{lemma}
\begin{proof}Consider
\begin{eqnarray*}(I + {w}{v}^T)(I - {w}{v}^T) & = & I - {w}{v}^T + {w}{v}^T - {w}{v}^T{w}{v}^T\\
& = & I - {w}{v}^T{w}{v}^T\\ & = & I\mbox{, because } {v}^T{w} = {w}^T{v} =
0,\quad \;\;\;\;\mbox{from Lemma \ref{lem-w^Tv = 0}.}\end{eqnarray*}
Therefore, $\left(I + {w}{v}^T\right)^{-1} = \left(I - {w}{v}^T\right).$\end{proof}
Taking the above result and substituting into (\ref{eq-alpha}), we obtain
\begin{eqnarray}\alpha_{j} & = &(L^T)^{-1}(I - {w}{v}^T)(\bar{U}^T)^{-1}{e}_j.\label{eq-alpha2}\end{eqnarray}
Next, we define an elementary matrix $E_{ij}$ by
\begin{eqnarray}E_{ij} & := & I - {e}_i{e}_i^T + {e}_i\alpha_j^T,\label{eq-E_BF}\end{eqnarray}
and note that it performs the desired row operation (\ref{eq-ipm row operation}) on A because
\begin{eqnarray}E_{ij}A & = & A - {e}_ir_i^T + {e}_i{e}_j^T,\nonumber\end{eqnarray}
in effect replacing the $i$-th row of $A$ with ${e}_j^T$. Therefore,
\begin{eqnarray}g_{ij}(P) & = & (-1)^{i+j}\det\left(A^{ij}\right) \;\;=\;\; \det(E_{ij}A) \;\;=\;\; \det(E_{ij})\det(A).\label{eq-g2}\end{eqnarray}
From (\ref{eq-E_BF}), we rewrite $E_{ij} = I - {e}_i({e}_i - \alpha_j)^T$. Then, from (\ref{eq-identity with c and d})
we obtain
\begin{eqnarray}\det(E_{ij}) & = & 1 - ({e}_i -
\alpha_j)^T{e}_i \;\;=\;\; 1 - {e}_i^T{e}_i + \alpha_j^T{e}_i \;\;=\;\; \alpha_j^T{e}_i.\label{eq-detE}\end{eqnarray}
Substituting (\ref{eq-alpha2}) into (\ref{eq-detE}) we obtain
\begin{eqnarray}\det(E_{ij}) & = & {e}_j^T(\bar{U})^{-1}(I - {v}{w}^T)(L)^{-1}{e}_i.\label{eq-detE_middle}\end{eqnarray}
For convenience we define $Q := I - {v}{w}^T$. Then
\begin{eqnarray}\det(E_{ij}) & = & a_j^TQb_i\mbox{,\quad where } \bar{U}^Ta_j
= {e}_j\mbox{ and }Lb_i = {e}_i.\label{eq-detE2}\end{eqnarray}
We now derive the main result of this subsection.
\begin{proposition}\label{prop-gradient} The general gradient element of $f(P)$ is given by
\begin{eqnarray}g_{ij}(P) = \frac{\partial f(P)}{\partial x_{ij}} & = & \det(A(P))(a_j^TQb_i).\label{eq-gradient_final}\end{eqnarray}\end{proposition}
\begin{proof}Substituting (\ref{eq-detE2}) into (\ref{eq-g2}) immediately yields the result.\end{proof}
\begin{remark}Note that we can calculate all $a_j$ and $b_i$ in advance, by solving the systems of linear
equations in (\ref{eq-detE2}), again in reduced row echelon form. Then, for the sake of efficiency we first calculate
\begin{eqnarray}\hat{q}_j^T & := & a_j^TQ,\quad j = 1, \hdots, N,\label{eq-q-hat_j}\end{eqnarray}
and then calculate
\begin{eqnarray}\hat{q}_{ij} & := & \hat{q}_j^Tb_i,\quad i = 1, \hdots, N,\quad j = 1, \hdots, N.\label{eq-q-hat_ij}\end{eqnarray}
This allows us to rewrite the formula for $g_{ij}(P)$ as
\begin{eqnarray}g_{ij}(P) = -f(P)\hat{q}_{ij}.\nonumber\end{eqnarray}\end{remark}
\subsection{Finding the Hessian matrix {\em H(P)}}\label{subsec-hessian ipm}
Here, we show that the LU decomposition found in Subsection \ref{subsec-det} can also be used to calculate the Hessian
of $f(P)$ efficiently. Consider $g_{ij}$ and $\hat{q}_{ij}$ as defined in (\ref{eq-gradient_final})
and (\ref{eq-q-hat_ij}) respectively. We prove the following result in this subsection:
\begin{eqnarray}H_{[ij],[k\ell]}(P) := \frac{\partial^2 f(P)}{\partial x_{ij} \partial x_{k\ell}} = g_{kj}\hat{q}_{i\ell} - g_{ij}\hat{q}_{k\ell},\nonumber\end{eqnarray}
where $H_{[ij],[k\ell]}$ is the general element of the Hessian matrix corresponding to arcs $(i,j)$ and
$(k,\ell)  \in \Gamma$.

The outline of the derivation is as follows.
\begin{enumerate}\item[(1)] We represent each element $H_{[ij],[k\ell]}(P)$ of the Hessian matrix as a cofactor of a minor of $A(P)$.
\item[(2)] We construct a second elementary matrix that in conjunction with $E_{ij}$ (see (\ref{eq-E_BF})) transforms matrix $A(P)$ into one with a determinant equivalent to the $(k,\ell)$-th cofactor of $A^{ij}(P)$.
\item[(3)] We then show that the general element of
the Hessian matrix is the product of $\det(A(P))$ and the determinants of the two elementary matrices.
\item[(4)] Using results obtained from finding $g(P)$ in Subsection \ref{subsec-gradient ipm}, we obtain these values immediately.\end{enumerate}
We define $A^{[ij],[k\ell]}$ to be the matrix $A$ with rows $i,k$ and columns $j,\ell$ removed. An argument similar to that for
$g_{ij}(P)$ in the previous subsection can be made that finding
\begin{eqnarray} H_{[ij],[k\ell]}(P) = \displaystyle\frac{\partial^2
f(P)}{\partial x_{ij} \partial x_{k\ell}} = (-1)^{(i+j+k+\ell+1)}\det(A^{[ij],[k\ell]})\mbox{, }i \neq k\mbox{, }j \neq \ell,\label{eq-H}\end{eqnarray}
is equivalent to finding the negative determinant of $A$ with the $i$th and $k$th rows changed to ${e}_j^T$ and ${e}_\ell^T$ respectively. That is,
\begin{eqnarray}\frac{\partial^2 f(P)}{\partial x_{ij} \partial x_{k\ell}} & = &
-\det(\hat{E}_{k\ell}E_{ij}A(P))\nonumber\\
& = & -\det(\hat{E}_{k\ell})\det(E_{ij})\det(A(P)),\label{eq-H2}\end{eqnarray}
where $\hat{E}_{k\ell}$ is an additional row operation constructed to change row $k$ of $E_{ij}A$ into ${e}_\ell^T$. Note
that if $i = k$ or $j = \ell$, the matrix $A^{[ij],[k\ell]}$ is no longer square and the determinant no longer exists.
If this occurs, we define $H_{[ij],[k\ell]} := 0$. If both $i = k$ and $j = \ell$, we also define $H_{[ij],[k\ell]} := 0$, as the determinant is linear in each element of $A(P)$.

Consider $E_{ij}A$ composed of rows $\hat{r}_1^T$, $\hat{r}_2^T$, $\hdots$, $\hat{r}_N^T$. Then, we
perform the following row operation:
\begin{eqnarray}\hat{r}_k & \rightarrow & \gamma_\ell(1)\hat{r}_1 + \gamma_\ell(2)\hat{r}_2 + \hdots + \gamma_\ell(N)\hat{r}_N,\label{eq-row operation for Hessian}\end{eqnarray}
where $(E_{ij}A)^T\gamma_\ell = {e}_\ell.$ Then, similarly to (\ref{eq-alpha}), we directly find $\gamma_\ell$:
\begin{eqnarray}\gamma_\ell & = & (E_{ij}^T)^{-1}(L^T)^{-1}(I - {w}{v}^T)(\overline{U}^T)^{-1}{e}_\ell.\label{eq-gamma_ell}\end{eqnarray}
Next, in a similar fashion to (\ref{eq-E_BF}), we construct an elementary matrix $\hat{E}_{k\ell}$
\begin{eqnarray}\hat{E}_{k\ell} & = & I - {e}_k{e}_k^T +{e}_k\gamma_\ell^T\nonumber\\
& = & I - {e}_k({e}_k^T - \gamma_\ell^T).\label{eq-Ehat}\end{eqnarray}
Then, we evaluate $\det(\hat{E}_{k\ell})$:
\begin{eqnarray}\det(\hat{E}_{k\ell}) & = & 1 - ({e}_k^T - \gamma_\ell^T){e}_k\nonumber\\
& = & 1 - 1 + \gamma_\ell^T{e}_k \nonumber\\ & = & {e}_\ell^T(\bar{U})^{-1}QL^{-1}(E_{ij})^{-1}{e}_k.\label{eq-detEhat}\end{eqnarray}
Recall from (\ref{eq-E_BF}) that $E_{ij} = I - {e}_i({e}_i^T - \alpha_j^T)$. We have
\begin{eqnarray}(E_{ij})^{-1} & = & I + \frac{1}{\alpha_j^T{e}_i}{e}_i({e}_i^T - \alpha_j^T).\label{eq-Einverse}\end{eqnarray}
Recall from (\ref{eq-g2}) that $g_{ij} = \det(A)\det(E_{ij})$, and from (\ref{eq-detE})
that $\alpha_j^T{e}_i = \det(E_{ij}) \neq 0$, and therefore (\ref{eq-Einverse}) holds. Then,
\begin{eqnarray}\alpha_j^T{e}_i & = & \displaystyle\frac{g_{ij}}{\det(A)}.\label{eq-alphaei}\end{eqnarray}
Substituting (\ref{eq-alphaei}) into (\ref{eq-Einverse}) we obtain
\begin{eqnarray}(E_{ij})^{-1} & = & I + \frac{\det(A)}{g_{ij}}{e}_i({e}_i^T - \alpha_j^T),\label{eq-Einverse2}\end{eqnarray}
and further substituting (\ref{eq-Einverse2}) into (\ref{eq-detEhat}), we obtain
\begin{eqnarray}\det(\hat{E}_{k\ell}) & = & {e}_\ell^T\left(\bar{U}\right)^{-1}QL^{-1}\left(I + \frac{\det(A)}{g_{ij}}{e}_i({e}_i^T - \alpha_j^T)\right){e}_k\nonumber\\ & = & {e}_\ell^T(\bar{U})^{-1}QL^{-1}\left({e}_k + \frac{\det(A)}{g_{ij}}{e}_i{e}_i^T{e}_k - \frac{\det(A)}{g_{ij}}{e}_i \alpha_j^T{e}_k\right).\label{eq-detEhat2}\end{eqnarray}
Note that since $i \neq k$, ${e}_i{e}_i^T{e}_k = {0}$, and from (\ref{eq-alphaei}),
$\alpha_j^T{e}_k = \displaystyle\frac{g_{kj}}{\det(A)}$. Hence, from (\ref{eq-detEhat2}) and
(\ref{eq-detE2}) we obtain
\begin{eqnarray}\det(\hat{E}_{k\ell}) & = & {e}_\ell^T(\bar{U})^{-1}QL^{-1}({e}_k
- \frac{g_{kj}}{g_{ij}}{e}_i) \nonumber\\ & = &a_\ell^TQ(b_k - \frac{g_{kj}}{g_{ij}} b_i).\label{eq-detEhat3}\end{eqnarray}
We now derive the main result of this subsection.
\begin{proposition}\label{prop-hessian}The general element of the Hessian of $f(P)$ is given by
\begin{eqnarray}H_{[ij],[k\ell]} & = & g_{kj}\hat{q}_{i\ell} - g_{ij}\hat{q}_{k\ell},\nonumber\end{eqnarray}
where $\hat{q}_{i\ell}$ and $\hat{q}_{k\ell}$ are defined in (\ref{eq-q-hat_ij}).\end{proposition}
\begin{proof}From (\ref{eq-H2}) and (\ref{eq-g2}), we can see that $H_{[ij],[k\ell]} = -\det(\hat{E}_{k\ell})g_{ij}$.
Then, from (\ref{eq-detEhat3}), $\det(\hat{E}_{k\ell}) = a_\ell^TQ(b_k -
\displaystyle\frac{g_{kj}}{g_{ij}}b_i)$ and
so $H_{[ij],[k\ell]} = - a_\ell^TQ(b_kg_{ij} - b_ig_{kj})$.

In order to improve computation time, we take advantage of the fact that we evaluate every $\hat{q}_{ij}$
while calculating the gradient to rewrite the second order partial derivatives of $f(P)$ as
\begin{eqnarray}H_{[ij],[k\ell]} & = & g_{kj}a_\ell^TQb_i - g_{ij}a_\ell^TQb_k \nonumber\\ & = & g_{kj}\hat{q}_{i\ell} - g_{ij}\hat{q}_{k\ell}.\label{eq-H3}\end{eqnarray}
This concludes the proof.\end{proof}
\begin{remark}Note that in practice, we do not calculate some $g_{kj}$'s when calculating $g(P)$ as an arc $(k,j)$
need not exist in the graph. In these cases we find $g_{jk}$ using the gradient formula,
$g_{jk} = -f(P)(\hat{q}_{jk})$, which remains valid despite arc $(k,j)$ not appearing in the graph.\end{remark}
\subsection{Leading principal minor}
It is, perhaps, interesting that instead of using the objective function $f(P) = -\det\left(I - P + \frac{1}{N} {e}{e}^T\right)$, it is also possible to use $f^1(P) := -\det(G^{NN}(P))$, the negative determinant of the
leading principal minor of $I - P$. The following, somewhat surprising, result justifies this claim.

\begin{theorem}\label{thm-lpm}{\begin{enumerate}\item[(1)]$f^1(P) = \displaystyle\frac{1}{N}f(P) = -\displaystyle\frac{1}{N}\det\left(I - P + \displaystyle\frac{1}{N}
{e}{e}^T\right)$.
\item[(2)] If the graph is Hamiltonian, then
\begin{eqnarray}\min\limits_{P \in \mathcal{DS}} f^1(P) = -1.\end{eqnarray}\end{enumerate}}\end{theorem}
\begin{proof}First, we show part (1), that is, $f^1(P) = \displaystyle\frac{1}{N}f(P)$. To find $f^1(P)$, we construct $LU = I - P$ as before, and define $\hat{L}$, $\hat{U}$ as:
\begin{eqnarray}\hat{L} = \left[\begin{array}{c}{e}_1^TL \\ \vdots \\ {e}_{N-1}^TL \\ {e}_N^T\end{array}\right],\quad \hat{U} = \left[\begin{array}{cccc} U{e}_1 & \cdots & U{e}_{N-1} &
{e}_N\end{array}\right].\label{eq-LhatUhat}\end{eqnarray}
That is, $\hat{L}$ is
the same as $L$ with the last row replaced by ${e}_N^T$, and $\hat{U}$ is the same as $U$ with the last
column replaced with ${e}_N$. Then consider
\begin{eqnarray}\hat{L}\hat{U} & = & \left[\begin{array}{c}{e}_1^TL\\ \vdots\\ {e}_{N-1}^TL\\ {e}_N^T\end{array}\right]\left[\begin{array}{cccc}U{e}_1 & \cdots & U{e}_{N-1} & {e}_N\end{array}\right]\nonumber\\ & = & \left[\begin{array}{cccc}{e}_1^TLU{e}_1 & \ddots & {e}_1^TLU{e}_{N-1} & {e}_1^TL{e}_N\\ \vdots & \ddots & \vdots & \vdots\\ {e}_{N-1}^TLU{e}_1 & \cdots
& {e}_{N-1}^TLU{e}_{N-1} & {e}_{N-1}^TL{e}_N\\ {e}_N^TU{e}_1 & \cdots & {e}_N^TU{e}_{N-1} & {e}_N^T{e}_N\end{array}\right].\nonumber\end{eqnarray}
Since $L$ is lower-triangular, ${e}_i^TL{e}_N = 0$ for all $i \neq N$. Likewise, since $U$ is
upper-triangular, ${e}_N^TU{e}_j = 0$ for all $j \neq N$. Therefore the above matrix simplifies to
\begin{eqnarray}\hat{L}\hat{U} & = & \left[\begin{array}{cccc}{e}_1^TLU{e}_1 & \cdots &
{e}_1^TLU{e}_{N-1} & 0\\ \vdots & \ddots & \vdots & \vdots\\ {e}_{N-1}^TLU{e}_1 & \cdots & {e}_{N-1}^TLU{e}_{N-1} & 0\\ 0 & \cdots & 0 & 1\end{array}\right],\nonumber\end{eqnarray}
which is the same as $LU$ with the bottom row and rightmost column removed, and a 1 placed in the bottom-right element.
Therefore, $\det(\hat{L}\hat{U}) = \det(G^{NN}(P))$, and consequently
\begin{eqnarray}f^1(P) & = &
-\det(\hat{L})\det({\hat{U}}).\label{eq-theta1}\end{eqnarray}
Note that $\hat{L}$ and $\hat{U}$ are triangular matrices, so\\
$$\det\left(\hat{L}\right) = \prod_{i = 1}^N
\hat{l}_{ii}\mbox{,\hspace*{0.7cm} and \hspace*{0.7cm}}\det\left(\hat{U}\right) = \prod_{i = 1}^N \hat{u}_{ii}.$$
However, only the last diagonal elements of $\hat{L}$ and $\hat{U}$ are different from $L$ and $\bar{U}$ (see (\ref{eq-lu2})) respectively, so
\begin{eqnarray}\det\left(\hat{L}\right) = \hat{l}_{NN}\prod_{i = 1}^{N-1}
l_{ii}\mbox{,\hspace*{0.7cm} and \hspace*{0.7cm}}\det\left(\hat{U}\right) = \hat{u}_{NN}\prod_{i = 1}^{N-1} \bar{u}_{ii}.\label{eq-detLhatdetUhat}\end{eqnarray}
Now, since $\hat{l}_{NN} = l_{NN} = 1$, we have
\begin{eqnarray}\det\left(\hat{L}\right) = \det\left(L\right) =
1.\label{eq-detLhat}\end{eqnarray}
We also have $\hat{u}_{NN} = 1$, but by Lemma \ref{lem-Ubar}, $\bar{u}_{NN} = N$ and hence
\begin{eqnarray}\det\left(\hat{U}\right) = \frac{1}{N}
\det\left(\bar{U}\right).\label{eq-detUhat}\end{eqnarray}
Therefore, substituting (\ref{eq-detLhat}) and (\ref{eq-detUhat}) into (\ref{eq-theta1}) we obtain
\begin{eqnarray*}f^1(P) & = &
-\det\left(\hat{L}\right)\det\left(\hat{U}\right)\\ & = & -\frac{1}{N}\det\left(\bar{U}\right)\\ & = &
-\frac{1}{N}\det\left(I - P + \frac{1}{N} {e}{e}^T\right) = \frac{1}{N}f(P).\end{eqnarray*}
Therefore, part (1) is proved.

The proof of part $(2)$ of Theorem \ref{thm-lpm} follows directly from the fact that
$\min f(P) = -N$ (proved in \cite{detpaper}),
and part $(1)$.\end{proof}
\vspace*{-0.4cm}\begin{remark}Using the leading principal minor has the advantage that the rank-one modification
$\frac{1}{N}{e}{e}^T$ is not required, which makes calculating the gradient and the Hessian even simpler than described in Subsection \ref{subsec-gradient ipm} and Subsection
\ref{subsec-hessian ipm} respectively. The derivation of the gradient and Hessian formulae for the negative determinant of the leading principal
minor follows the same process as that for the determinant function, except that the matrix $Q = I - {v}{w}^T$
is not required. \end{remark}
The formulae for $f^1(P)$, $g^1(P)$ and $H^1(P)$ then reduce to
\begin{eqnarray}f^1 & = & -\prod_{i = 1}^{N-1} u_{ii},\label{eq-d1}\\
g^1_{ij} & = & -f^1(P)(a_j^1)^Tb_i^1,\label{eq-g1}\\ H^1_{[i,j],[k,\ell]} & = & g^1_{kj}(a_\ell^1)^Tb_i^1 -
g^1_{ij}(a_\ell^1)^Tb_k^1,\label{eq-H1}\end{eqnarray}
where
\begin{eqnarray} \hat{L}b_i^1 & = & {e}_i,\label{eq-Lhat}\\
\hat{U}^Ta_j^1 & = & {e}_j.\label{eq-Uhat}\end{eqnarray}
\vspace*{-1.2cm}\begin{remark}In practice, the determinant of the leading principal minor is used rather than
that of the whole matrix. It is simpler, more efficient and the optimal value is independent of the
graph. It eliminates the need to scale any parameters by the size of the graph. When $f^1(P)$ is used in lieu of $f(P)$ the corresponding gradient vector and Hessian matrix are denoted by $g^{1}(P)$ and $H^{1}(P)$, respectively.\end{remark}
\vspace*{-0.8cm}
\section{LU decomposition-based evaluation algorithm}\label{sec-lu decomposition algorithm}
\vspace*{-0.35cm}

The algorithm for computing $f^1(P)$, $g^1(P)$, $H^1(P)$ is given here, along with the complexity of each step of the algorithm. Let $k$ denote the average degree of the graph, that is, there are $kN$ edges.
\vspace*{-0.5cm}
{\scriptsize\begin{center}\begin{tabular}{|llc|}\hline
{\bf Input}: $P$& &\\
{\bf Output}: $f^1(P), g^1(P), H^1(P)$& &\\
& &\\
{\bf begin}& & {\bf \underline{Complexity}}\\
\hspace*{1cm}1) Perform LU decomposition to find $LU = I - P$. & & $O(kN^2)$\\
& &\\
\hspace*{1cm}2) Calculate $\hat{L}$ and $\hat{U}$, using (\ref{eq-LhatUhat}).& & $O(N)$\\
& &\\
\hspace*{1cm}3) Calculate each $(a_j^1)^T$ and $b_i^1$, using (\ref{eq-Lhat}) and (\ref{eq-Uhat}).& & $O(N^3)$\\
& &\\
\hspace*{1cm}4) Calculate each $(a_j^1)^Tb_i^1$.& & $O(N^3)$\\
& &\\
\hspace*{1cm}5) Calculate $f^1(P) = -\prod_{i=1}^{N-1} u_{ii}$.& & $O(N)$\\
& &\\
\hspace*{1cm}6) Calculate each $g_{ij}^1(P) = -f^1(P)(a_j^1)^Tb_i^1$.& & $O(kN)$\\
& &\\
\hspace*{1cm}7) Calculate each $H^1_{[ij],[k\ell]}(P) = \left\{\begin{array}{ccl}g_{kj}^1(a_\ell^1)^Tb_i^1 - g_{ij}^1(a_\ell^1)^Tb_k^1, &  & i \neq k\mbox{ and } j \neq \ell\mbox{ and } i,j,k,l \neq N\\
0, &  & \mbox{otherwise.}\end{array}\right.$& & $O(k^2N^2)$\\
{end}& &\\
\hline\end{tabular}
Function evaluations algorithm
\end{center}}
\vspace*{-0.5cm}
If the graph is sparse, the complexity of the above algorithm is $O(N^3)$. However, for sufficiently dense graphs (that is, $k > \sqrt{N}$) the complexity of the above algorithm is $O(k^2N^2)$. Note that each element of the Hessian is calculated in $O(1)$ time, because they simply involve scalar multiplication
where all of the scalars have already been calculated in earlier steps, that is, the gradient terms in step 6, and
each $\left(a_i^1\right)^Tb_l^1$ in step 4.

These bounds are considerably better than the $O(k^3N^4)$ bound that applies if we simply perform an LU decomposition for each element in the Hessian and gradient.
\newpage
\begin{example}Consider the following six-node cubic graph $\Gamma_6$.
\vspace*{0.1cm}

\begin{figure}[h]\begin{center}\includegraphics[scale=0.8]{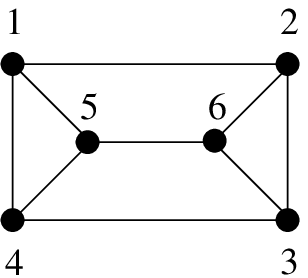}
\end{center}\label{fig-envelope}\end{figure}

The adjacency matrix of $\Gamma_6$ is
$${\scriptsize\left[\begin{array}{cccccc} 0 & 1 & 0 & 1 & 1 & 0\\
1 & 0 & 1 & 0 & 0 & 1\\ 0 & 1 & 0 & 1 & 0 & 1\\ 1 & 0 & 1 & 0 & 1 & 0\\ 1 & 0 & 0 & 1 & 0 & 1\\ 0 & 1 & 1 & 0 & 1 &
0\end{array}\right]}.$$

Consider a point ${x}$ such that,\\
$$P({x}) = {\scriptsize\left[\begin{array}{cccccc}0 & \frac{2}{3} & 0 & \frac{1}{6} & \frac{1}{6} & 0\\ \frac{2}{3} & 0 &
\frac{1}{6} & 0 & 0 & \frac{1}{6}\\ 0 & \frac{1}{6} & 0 & \frac{2}{3} & 0 & \frac{1}{6}\\ \frac{1}{6} & 0 & \frac{2}{3}
& 0 & \frac{1}{6} & 0\\ \frac{1}{6} & 0 & 0 & \frac{1}{6} & 0 & \frac{2}{3}\\ 0 & \frac{1}{6} & \frac{1}{6} & 0 &
\frac{2}{3} & 0\end{array}\right]}.$$\\
\\
Performing the LU decomposition of $I - P$ using MATLAB's {\texttt{lu}} routine we obtain matrices $L$ and $U$ (given to four decimal places)
$$\hspace*{-0.45cm}L = {\scriptsize\left[\begin{array}{rrrrrr}1 & 0 & 0 & 0
& 0 & 0\\ -0.6667 & 1 & 0 & 0 & 0 & 0\\ 0 & -0.3000 & 1 & 0 & 0 & 0\\ -0.1667 & -0.2000 & -0.7368 & 1 & 0 &
\,\quad \;\;\;\;\;0\\ -0.1667 & -0.2000 & -0.0351 & -0.5556 & 1 & 0\\ 0 & -0.3000 & -0.2281 & -0.4444 & -1.0000 &
1\end{array}\right],} \hspace*{0.25cm}
U = {\scriptsize\left[\begin{array}{rrrrrr} 1 & -0.6667 & 0 & -0.1667 & -0.1667 & 0\\ \,\quad \;\;\;\;\;0 &
0.5556 & -0.1667 & -0.1111 & -0.1111 & -0.1667\\ 0 & 0 & 0.9500 & -0.7000 & -0.0333 & -0.2167\\ 0 & 0 & 0 & 0.4342 &
-0.2412 & -0.1930\\ 0 & 0 & 0 & 0 & 0.8148 & -0.8148\\ 0 & 0 & 0 & 0 & 0 & 0\end{array}\right].}$$
Consequently, $\hat{L} = \left[\begin{array}{cccc}L^T{e}_1 & \cdots & L^T{e}_{N-1} & {e}_N\end{array}\right]^T$ and
$\hat{U} = \left[\begin{array}{cccc} U{e}_1 & \cdots & U{e}_{N-1} & {e}_N\end{array}\right]$ are simply

\vspace*{-0.65cm}$$\hspace*{-0.45cm}\hat{L} = {\scriptsize\left[\begin{array}{rrrrrr}1 & 0 & 0 & 0
& 0 & 0\\ -0.6667 & 1 & 0 & 0 & 0 & 0\\ 0 & -0.3000 & 1 & 0 & 0 & 0\\ -0.1667 & -0.2000 & -0.7368 & 1 & 0 &
\,\quad \;\;\;\;\;0\\ -0.1667 & -0.2000 & -0.0351 & -0.5556 & \,\quad \;\;\;\;\;1 & 0\\ 0 & 0 & 0 & 0 & 0 &
1\end{array}\right],} \hspace*{0.25cm} \hat{U} = {\scriptsize\left[\begin{array}{rrrrrr} 1 & -0.6667 & 0 & -0.1667 & -0.1667 & 0\\ \,\quad \;\;\;\;\;0
& 0.5556 & -0.1667 & -0.1111 & -0.1111 & 0\\ 0 & 0 & 0.9500 & -0.7000 & -0.0333 & 0\\ 0 & 0 & 0 & 0.4342 & -0.2412 & 0\\
0 & 0 & 0 & 0 & 0.8148 & 0\\ 0 & 0 & 0 & 0 & 0 & \,\quad \;\;\;\;\;1\end{array}\right].}$$
For all $i,j$, we calculate the $a_j^1$ and $b_i^1$ vectors using (\ref{eq-Lhat}) and (\ref{eq-Uhat}). Namely,

\vspace*{-0.45cm}$$a_1^1 = {\scriptsize\left[\begin{array}{c}1\\
1.2\\ 0.2105\\ 1.0303\\ 0.6818\\ 0\end{array}\right]}, a_2^1 = {\scriptsize\left[\begin{array}{c}0\\ 1.8\\
0.3158\\ 0.9697\\ 0.5455\\ 0\end{array}\right]}, a_3^1 = {\scriptsize\left[\begin{array}{c}0\\ 0\\ 1.0526\\ 1.6970\\ 0.5455\\ 0\end{array}\right]}, a_4^1 = {\scriptsize\left[\begin{array}{c}0\\ 0\\ 0\\ 2.3030\\ 0.6818\\
0\end{array}\right]}, a_5^1 = {\scriptsize\left[\begin{array}{c}0\\ 0\\ 0\\ 0\\ 1.2273\\
0\end{array}\right]}, a_6^1 = {\scriptsize\left[\begin{array}{c}0\\ 0\\ 0\\ 0\\ 0\\ 1\end{array}\right]},$$
$$b_1^1 = {\scriptsize\left[\begin{array}{c}1\\ 0.6667\\ 0.2\\ 0.4474\\ 0.5556\\ 0\end{array}\right]}, b_2^1 =
{\scriptsize\left[\begin{array}{c}0\\ 1\\ 0.3\\ 0.4211\\ 0.4444\\ 0\end{array}\right]}, b_3^1 =
{\scriptsize\left[\begin{array}{c}0\\ 0\\ 1\\ 0.7368\\ 0.4444\\ 0\end{array}\right]}, b_4^1 =
{\scriptsize\left[\begin{array}{c}0\\ 0\\ 0\\ 1\\ 0.5556\\ 0\end{array}\right]}, b_5^1 =
{\scriptsize\left[\begin{array}{c}0\\ 0\\ 0\\ 0\\ 1\\ 0\end{array}\right]}, b_6^1 =
{\scriptsize\left[\begin{array}{c}0\\ 0\\ 0\\ 0\\ 0\\ 1\end{array}\right]}. $$
We can now represent each $(a_j^1)^Tb_i^1$ as the $ij$-th element of the matrix
$$\left[(a_j^1)^T(b_i^1)\right]_{i,j=1}^{N,N} = {\scriptsize\left[\begin{array}{cccccc}\;\;2.6818\;\; & 2 & \;\;1.2727\;\; & \;\;1.4091\;\; & \;\;0.6818\;\; & \;0\;\\
2 & \;\;2.5455\;\; & 1.2727 & 1.2727 & 0.5455 & \;0\;\\
1.2727 & 1.2727 & 2.5455 & 2 & 0.5455 & 0\\
1.4091 & 1.2727 & 2 & 2.6818 & 0.6818 & 0\\
0.6818 & 0.5455 & 0.5455 & 0.6818 & 1.2273 & 0\\
0 & 0 & 0 & 0 & 0 & 1\end{array}\right].}$$
Then, $f^1(P) = \prod\limits_{i=1}^{N-1} \hat{u}_{ii} \approx -0.1867$. Note that we can directly verify the preceding by confirming that $\det(A(P)) \approx 1.1204 = -6(f^1(P))$.

The gradient vector is then found using (\ref{eq-g1}). Note that we are only interested in the gradient elements for the eighteen arcs in the graph; this yields, to three decimal places:
$$g^1(P) \approx
{\scriptsize\left[\begin{array}{cccccccccccccccccc}\;0.374\; & \;0.263\; & \;0.127\; & \;0.374\; & \;0.238\; & \;0\; & \;0.238\; & \;0.374\; & \;0\; & \;0.263\; & \;0.374\; & \;0.127\; & \;0.127\; & \;0.127\; & \;0\; & \;0\; & \;0\; & \;0\;\end{array}\right]}.$$
Finally, the Hessian is found using (\ref{eq-H1}), given here to two decimal places:
$$H^1(P) \approx {\scriptsize\left[\begin{array}{cccccccccccccccccc}0 & 0 & 0 & 0.53 & 0.13 & 0 & 0 & -0.41 & 0 & 0.11 & -0.44 & -0.09 & 0.02 & -0.11 & 0 & 0 & 0 & 0 \\
0 & 0 & 0 & 0.11 & -0.03 & 0 & 0.41 & 0 & 0 & 0.97 & 0.11 & 0.16 & 0.16 & 0 & 0 & 0 & 0 & 0 \\
0 & 0 & 0 & 0.02 & -0.03 & 0 & 0.04 & -0.11 & 0 & 0.16 & -0.09 & 0 & 0.53 & 0.24 & 0 & 0 & 0 & 0 \\
0.53 & 0.11 & 0.02 & 0 & 0 & 0 & 0.13 & -0.44 & 0 & 0 & -0.41 & -0.11 & 0 & -0.09 & 0 & 0 & 0 & 0 \\
0.13 & -0.03 & -0.03 & 0 & 0 & 0 & 0.91 & 0.13 & 0 & 0.41 & 0 & 0.04 & 0.04 & -0.03 & 0 & 0 & 0 & 0 \\
0 & 0 & 0 & 0 & 0 & 0 & 0 & 0 & 0 & 0 & 0 & 0 & 0 & 0 & 0 & 0 & 0 & 0 \\
0 & 0.41 & 0.04 & 0.13 & 0.91 & 0 & 0 & 0 & 0 & -0.03 & 0.13 & -0.03 & -0.03 & 0.04 & 0 & 0 & 0 & 0 \\
-0.41 & 0 & -0.11 & -0.44 & 0.13 & 0 & 0 & 0 & 0 & 0.11 & 0.53 & 0.02 & -0.09 & 0 & 0 & 0 & 0 & 0 \\
0 & 0 & 0 & 0 & 0 & 0 & 0 & 0 & 0 & 0 & 0 & 0 & 0 & 0 & 0 & 0 & 0 & 0 \\
0.11 & 0.97 & 0.16 & 0 & 0.41 & 0 & -0.03 & 0.11 & 0 & 0 & 0 & 0 & 0 & 0.16 & 0 & 0 & 0 & 0 \\
-0.44 & 0.11 & -0.09 & -0.41 & 0 & 0 & 0.13 & 0.53 & 0 & 0 & 0 & 0 & -0.11 & 0.02 & 0 & 0 & 0 & 0 \\
-0.09 & 0.16 & 0 & -0.11 & 0.04 & 0 & -0.03 & 0.02 & 0 & 0 & 0 & 0 & 0.24 & 0.53 & 0 & 0 & 0 & 0  \\
0.02 & 0.16 & 0.53 & 0 & 0.04 & 0 & -0.03 & -0.09 & 0 & 0 & -0.11 & 0.24 & 0 & 0 & 0 & 0 & 0 & 0  \\
-0.11 & 0 & 0.24 & -0.09 & -0.03 & 0 & 0.04 & 0 & 0 & 0.16 & 0.02 & 0.53 & 0 & 0 & 0 & 0 & 0 & 0 \\
0 & 0 & 0 & 0 & 0 & 0 & 0 & 0 & 0 & 0 & 0 & 0 & 0 & 0 & 0 & 0 & 0 & 0  \\
0 & 0 & 0 & 0 & 0 & 0 & 0 & 0 & 0 & 0 & 0 & 0 & 0 & 0 & 0 & 0 & 0 & 0  \\
0 & 0 & 0 & 0 & 0 & 0 & 0 & 0 & 0 & 0 & 0 & 0 & 0 & 0 & 0 & 0 & 0 & 0  \\
0 & 0 & 0 & 0 & 0 & 0 & 0 & 0 & 0 & 0 & 0 & 0 & 0 & 0 & 0 & 0 & 0 & 0 \end{array}\right].}$$

\end{example}

\vspace*{-0.5cm}\begin{acknowledgements}
Support for this work was provided by Australian Research Council (DP0666632 and DP0984470), the Office of Naval Research (Grant: N00014-02-1-0076) and the Army (Grant:  W911NF-07-2-0027-1). We would also like to thank V. Ejov for useful discussions and the referees whose comments corrected some errors and prompted us to improve the presentation.
\end{acknowledgements}

\end{document}